\documentclass[12pt]{amsart}
\usepackage{amsmath,amssymb,amsthm,hyperref}
\usepackage{graphicx}
\usepackage{tikz}  
 \newtheorem{theorem}{Theorem}

 \newtheorem{lemma}[theorem]{Lemma}
 \newtheorem{proposition}[theorem]{Proposition}
 
 {\theoremstyle{definition}
 \newtheorem{definition}{Definition}
 \newtheorem{remark}{Remark}

 }

\newcommand{\blue}{\color{black}}
\newcommand{\magenta}{\color{black}}
\newcommand{\black}{\color{black}}

\newcommand{\green}{\color{black}}

\def\RR{\mathbb{R}}\def\R{\mathbb{R}}
\def\ZZ{\mathbb{Z}}
\def\NN{\mathbb{N}}

\def\PP{\mathbb{P}}
\def\EE{\mathbb{E}}\def\E{\mathbb{E}}
\def\Var{\mathbb{V}\mathrm{ar}}
\def\Cov{\mathbb{C}\mathrm{ov}}

\def\ga{\gamma}
\def\de{\delta}

\def\ph{\varphi}

\def\wt{\widetilde}

\def\dtv{d_{\mathrm{TV}}}

\newcommand{\norm}[1]{\left\lVert#1\right\rVert}

\renewcommand{\qed}{\hfill$\square$}

\DeclareMathOperator\dom{dom} 

\title[Rate for the Breuer-Major theorem]{\black \small Berry-Esseen bounds in the Breuer-Major CLT and Gebelein's inequality}

\author{Ivan Nourdin}
\address{Ivan Nourdin, Universit\'e du Luxembourg, 
Unit\'e de Recherche en Math\'ematiques,
Maison du Nombre,
6 avenue de la Fonte,
L-4364 Esch-sur-Alzette,
Grand Duch\'e du Luxembourg}
\email{ivan.nourdin@uni.lu} 
\thanks{}
\author{Giovanni Peccati}
\address{Giovanni Peccati, Universit\'e du Luxembourg, 
Unit\'e de Recherche en Math\'ematiques,
Maison du Nombre,
6 avenue de la Fonte,
L-4364 Esch-sur-Alzette,
Grand Duch\'e du Luxembourg}
\email{giovanni.peccati@uni.lu} 
\thanks{}
\author{Xiaochuan Yang}
\address{Xiaochuan Yang, Universit\'e du Luxembourg, 
Unit\'e de Recherche en Math\'ematiques,
Maison du Nombre,
6 avenue de la Fonte,
L-4364 Esch-sur-Alzette,
Grand Duch\'e du Luxembourg}
\email{xiaochuan.yang@uni.lu} 
\thanks{}

\date{\today}
\begin{document}
\maketitle

\medskip

\begin{abstract}

{\black We derive explicit Berry-Esseen bounds in the total variation distance for the Breuer-Major central limit theorem, in the case of a subordinating function $\varphi$ satisfying minimal regularity assumptions. Our approach is based on the combination of the Malliavin-Stein approach for normal approximations with Gebelein's inequality, bounding the covariance of functionals of Gaussian fields in terms of maximal correlation coefficients.}\\

\noindent{{\bf Keywords: } Breuer-Major theorem, rate of convergence, Gebelein's inequality, Malliavin-Stein approach.
}
\end{abstract}

\section{Introduction}

\subsection{Motivation and main results}

Let $X=(X_k)_{k\in\ZZ}$ be a centered stationary Gaussian sequence with covariance function $\EE[X_kX_j]=\rho(k-j)$ satisfying $\rho(0)=1$. Let $\ph\in L^2(\RR,\ga)$, where $\ga$ is the standard Gaussian measure on the real line, and assume without loss of generality that $\EE[\varphi(X_1)]=\int_\R \varphi d\gamma=0$.  {\black By exploiting the orthogonality and completeness of Hermite polynomials in} $L^2(\RR,\ga)$ (see, e.g., \cite[p.~13]{NPbook}), we can write 
\begin{equation}\label{e:he}
\ph = \sum_{\ell\ge d}a_\ell H_\ell,
\end{equation}
where $H_\ell$ is the Hermite polynomial of order $\ell$, the coefficient $a_d$ is different from zero, $d\ge 1$ is the {\it Hermite rank} of $\ph$, and the series converges in $L^2(\RR,\ga)$. 
Consider the {\black sequence of} normalized sums
\begin{align}\label{e:F_n}
F_n =  \frac{1}{\sqrt{n}} \sum_{k=1}^n \ph(X_k), \quad {\black n\geq 1}.
\end{align}
{\black The celebrated {\it Breuer-Major theorem} \cite{BM}, stated below, provides sufficient conditions on the covariance function $\rho$, in order for $F_n$ to exhibit Gaussian fluctuations, as $n\to \infty$} (see also Taqqu \cite{T} for a related work). {\black Throughout the paper, the symbol $N(a, b)$ denotes a Gaussian random variable with mean $a\in \R$ and variance $b\geq 0$}, {\green and $\overset{d}\to $ the convergence in distribution. }
\begin{theorem}[Breuer-Major Theorem]\label{BMtheorem}
Let the previous assumptions on $X$ and $\varphi$ prevail,
and suppose moreover that
 $\sum_{k\in\ZZ} |\rho(k)|^d <\infty$. Then  $F_n\overset{d}\to N(0,\sigma^2)$, where
\begin{align}\label{e:sigma}
\sigma^2 = \sum_{\ell\ge d} a_\ell^2 \ell! \sum_{k\in\ZZ} \rho(k)^\ell<\infty.
\end{align}
Here, and for the rest of the paper, $\overset{d}\rightarrow$ denotes convergence in distribution {\black of random variables}. 
\end{theorem}
{\blue The Breuer-Major theorem has far-reaching applications in many different areas, such as mathematical statistics, signal processing or geometry of random nodal sets, see e.g. \cite{Doukhan,NourdinPeccatiRossi,Pipiras-Taqqu,Tudor} and references therein. It has been generalized and refined in various aspects \cite{CNN,CS,NN,NP,NPP}. }

\medskip

{\black Now let $\sigma_n^2  := \Var (F_n)$ and $V_n := F_n/\sqrt{\Var(F_n)}$. The aim of the present paper is develop a novel method for obtaining explicit upper bounds on the sequence
$$
d_{\rm TV}(V_n , N(0,1)) := \sup_{A\in \mathcal{B}(\R)} \left| \mathbb{P}(V_n \in A) - \mathbb{P}(N(0,1) \in A)\right|, \quad n\geq 1,
$$
{\green where $\mathcal B(\R)$ is the Borel $\sigma$-algebra on $\RR$}, under {\it minimal regularity assumptions} on the function $\ph$. {\magenta Our strategy for doing so is to combine the {\it Malliavin-Stein method for probabilistic approximations} (as described in Section \ref{ss:ms} below) and the powerful {\it Gebelein's inequality} for correlation of Gaussian functionals (see \cite{G,V}, as well as Section \ref{s:pg}, for a self-contained proof), as applied to non-linear transformations of correlated Gaussian sequences. To the best of our knowledge, our use of Gebelein's inequality is new: it is reasonable to expect that the content of the present work might constitute the blueprint for further applications of such general a bound to probabilistic approximations in a Gaussian setting.}

\smallskip

We recall that, for every $n$, the quantity $d_{\rm TV}(V_n , N(0,1))$ corresponds to the {\it total variation distance} between the distributions of $V_n$ and $N(0,1)$ --- see e.g. \cite[Appendix C]{NP}, and the references therein, for a discussion of the properties of $d_{\rm TV}$. Any statement yielding the existence of an explicit numerical sequence $\{\alpha_n\}$ such that $\alpha_n\to 0$ and $d(V_n , N(0,1))\leq \alpha_n$, {\black for some distance $d$}, is called a {\it quantitative Breuer-Major Theorem}.} 

\smallskip

{\black One of the first quantitative Breuer-Major theorems is contained in the work by Nourdin, Peccati and Podolskij \cite{NPP}  --- see, in particular, \cite[Cor.~2.4]{NPP}, where the focus is on the {\black Kolmogorov and 1-Wasserstein distances} and on the case where $\ph$ is a Hermite polynomial of order $q$. The rates obtained in \cite{NPP} are, in general, not optimal. }{\black We stress that, according to \cite[Corollary 2.4]{NPP}, the convergence in distribution in Theorem \ref{BMtheorem} always takes place in the sense of the Kolmogorov and 1-Wasserstein distances. }

\medskip

Determining whether the Breuer-Major CLT holds in the topology of the distance $d_{\rm TV}$ is a much more delicate matter, since -- unlike convergence in the Kolmogorov or 1-Wasserstein distances -- convergence in total variation cannot take place in full generality, and requires extra regularity assumptions on $\varphi$. Our specific aim is therefore to tackle the following problem:
 
 \bigskip
 
 \noindent {\bf Problem P}: Letting the notation and assumptions of Theorem \ref{BMtheorem} prevail, find  conditions on $\varphi$ and $\rho$ in order to have that $$d_{\rm TV}(F_n/\sqrt{\Var(F_n)} , N(0,1)) \to 0\quad\mbox{as $n\to\infty$}.$$

\medskip

To appreciate the subtlety of Problem P, one should recall the following two facts:
\begin{itemize}

\item[(i)] according to the main findings in \cite{NourdinPoly}, if $\varphi$ is a polynomial, then the convergence in Theorem \ref{BMtheorem} always takes place in the sense of total variation;

\item[(ii)] on the other hand, if one considers independent $X_k\sim  N(0,1)$, then it is immediate to build counterexamples, for instance by setting $\varphi(x)={\rm sign}(x)$ --- in which case the assumptions of Theorem \ref{BMtheorem} are satisfied, but $d_{\rm TV}(F_n/\sqrt{\Var(F_n)} , N(0,1))=1$ for all $n$. 

\end{itemize} 

{\magenta As anticipated, the content of Points (i) and (ii) suggests that there exists a {\it minimal} amount of regularity for the function $\varphi$, below which convergence in total variation in the Breuer-Major Theorem ceases to take place. Exactly locating such a threshold is the ultimate goal of the line of research inaugurated by the present paper.} 

\smallskip

As already discussed, in what follows we will be concerned with {\black upper bounds on} the rate of convergence in the Breuer-Major theorem when the function $\ph$ possibly displays an {\it infinite} Hermite expansion \eqref{e:he}, and {\black belongs to the Sobolev space $\mathbb{D}^{1,4}$ --- where we adopted the usual notation $\mathbb D^{p,q}$ in order to indicate the Sobolev space of {\black those} random variables on {\black a Gaussian} space that are $p$ times differentiable in the sense of Malliavin, and whose Malliavin derivative is $q$-integrable (see Section \ref{s:2} for a precise definition).  We consider that the property of belonging to some space $\ph\in\mathbb{D}^{1,q}$, $q\geq 1$, is {\magenta somehow unavoidable}, in the sense that it is the least requirement on $\varphi$ that allows one to directly apply the Malliavin-Stein method outlined in Section 2.}

\medskip

The following statement is the main result of the paper:

\begin{theorem}\label{t:main}
Let $X=(X_k)_{k\in\ZZ}$ be a centered stationary Gaussian sequence with covariance function $\EE[X_kX_j]=\rho(k-j)$ satisfying $\rho(0)=1$, and let $\ph\in \mathbb{D}^{1,4}\subset L^2(\RR,\ga)$
be such that $\EE[\varphi(X_1)]=\int_\R \varphi d\gamma=0$. 
Let $F_n$ be given by \eqref{e:F_n} and set $\sigma_n^2=\Var(F_n)$ and $V_n = F_n/\sigma_n$. 
Then, for a finite constant $C(\ph)$, whose {\black explicit} value is given in \eqref{e:C(ph)} below:
\begin{enumerate}
\item[(i)] For every $n$,
\begin{align}\label{e:1}
\dtv(V_n, N(0,1)) \le \frac{4C(\ph)}{\sigma_n^2} n^{-\frac 1 2} \left(\sum_{|k|<n} |\rho(k)| \right)^{\frac{3}{2}}.
\end{align}
\item[(ii)]  If $\ph$ is symmetric (or, more generally, $2$-sparse, as defined in Section \ref{s:sparse}) then, for all $b\in[1,2]$ and all $n$,
\begin{eqnarray}
&& \dtv(V_n, N(0,1)) \label{e:2} \\
&& \quad\quad\quad\quad\le \frac{4C(\ph)}{\sigma_n^2}  n^{-(\frac{1}{b}-\frac{1}{2})} \left(\sum_{|k|<n} |\rho(k)|^2\right)^\frac{1}{2} \left(\sum_{|k|<n} |\rho(k)|^{b}\right)^{1\over b} .\notag
\end{eqnarray}
\end{enumerate}
\end{theorem}

\begin{remark}\label{r:r} (1) {\black  {\black Recall that, according e.g. to the terminology adopted in \cite[Chapter 9]{NPbook}, a numerical sequence $\alpha_n\downarrow 0$ is said to provide an {\it optimal rate} (for $\dtv(V_n, N(0,1))$), whenever there exist non-zero finite constants $k<K$ such that
$$
k \alpha_n \leq \dtv(V_n, N(0,1)) \leq K\alpha_n,
$$
for $n$ large enough. }The rate provided in Theorem \ref{t:main}-(i) for functions $\varphi$ with Hermite rank 1 is optimal in this sense. Indeed, in the trivial case where $\rho(j) = 0$ for every $j\neq 0$ and using e.g. the {\it reverse Berry-Esseen inequality} from \cite{BH}, it is easy to build a centered smooth function $\varphi$ with Hermite rank 1 and such that $$ d_{\rm TV}(V_n, N(0,1)) \geq C n^{-1/2},$$ for some absolute constant $C>0$.

(2) For a function $\varphi$ having Hermite rank equal to 2, the sufficient condition for asymptotic normality in Theorem \ref{BMtheorem} is that $\rho\in\ell^2(\ZZ)$. Theorem \ref{t:main}-(ii) refines such a result by yielding that, in the case of a symmetric $\varphi$, convergence in total variation takes place whenever $\rho\in \ell^b$ ($\subset \ell^2$), for some $b\in[1,2)$.   We also observe that Theorem \ref{t:main}-(ii) yields an upper bound on $\dtv(V_n, N(0,1))$, {\black explicitly} interpolating all the cases $\rho\in\ell^b(\ZZ)$, for $1\le b<2$. 

(3) {\black By inspection of our forthcoming proof, it will be clear that our techniques do not allow us to deal with the case of a general function $\varphi \in \mathbb D^{1,4}$ having Hermite rank equal to 2. This implies that the requirement that $\varphi$ is 2-sparse cannot easily be removed.}}
\end{remark}


We will now compare our findings with further results in the literature. 

\subsection{Discussion}

In the case where $\ph$ has a possibly infinite Hermite expansion \eqref{e:he}, and under some extra smoothness assumptions, Nourdin, Peccati and Reinert \cite{NPR}, Nualart and Zhou \cite{NZ} and Vidotto \cite{Vidotto} obtained total variation error bounds that are better than those derived in \cite{NPP}. {\black The rates of convergence deduced in \cite{NPP} and \cite{ NPR, NZ, Vidotto} (that are sometimes optimal, and sometimes not) are all obtained via some variation of the {Malliavin-Stein approach} described in Section \ref{ss:ms}.} 

\smallskip

In \cite{NZ} (the closest reference to the present note), the following general quantitative result is proved (see \cite[Th.~4.2 and Th.~4.3(v)]{NZ}):  {\black as $n\to\infty$, one has that $$d_{\rm TV}(V_n , N(0,1)) = O(n^{-1/2}),$$ provided that

\begin{itemize}

\item[(a)] either $\ph$ has Hermite rank 1 ($d=1$ in \eqref{e:he}), $\ph\in\mathbb D^{2,4}$ and $\rho\in \ell^{1}(\ZZ)$, or

\item[(b)] $\ph$ has Hermite rank 2 ($d=2$ in \eqref{e:he}), $\ph\in\mathbb D^{6,8}$ and $\rho\in \ell^{\frac 3 2}(\ZZ)\subset \ell^2(\ZZ)$.

\end{itemize}

}

{\black The regularity assumptions on $\varphi$ required at Points (a) and (b) above are clearly more restrictive than ours. On the other hand, disregarding the regularity of $\varphi$, the upper bound of the order $n^{-1/2}$ obtained in \cite{NZ} is optimal for the set of assumptions at Point (a) and (b) above}. The optimality for Point (a) follows from the same argument used in Remark \ref{r:r}-(1). Similarly, the order $n^{-1/2}$ under the set of assumptions at Point (b) cannot be improved in general, since it coincides with the {\it third/fourth cumulant barrier} for the total variation distance, between {\black the laws of} a sequence of random variables in a fixed chaos and the standard normal distribution. {\black Such a result was established in full generality in \cite[Theorem 11.2]{NP}, and is presented in the next proposition in the simple case of polynomials of order 2.  }   {\green Here and after, $a(n)\asymp b(n)$ means that the ratio $a(n)/b(n)$ is bounded from above and below by positive finite constants.}

\begin{proposition}\cite[Proposition 4.2]{NP}\label{p:pr}
Let $F_n$ be given by \eqref{e:F_n} with $\ph=H_2$. Set $V_n = F_n/\sqrt{\Var(F_n)}$.  Then,
\begin{align*}
\dtv(V_n,N(0,1))\asymp \frac{1}{\sqrt{n}} \left(\sum_{|k|<n} |\rho(k)|^{\frac 3 2} \right)^2
\end{align*}
as $n\to \infty$.  In particular, $\dtv(V_n,N(0,1))\asymp \frac{1}{\sqrt{n}}$ if $\rho\in\ell^{3\over 2}(\ZZ)$.
\end{proposition}

{\blue One interesting subordinating function $\varphi$ entering the scope of our paper is  $\ph(x)=|x|-\sqrt{2/\pi}$. The Breuer-Major CLT associated with such a mapping has been recently applied in a geometric setting in \cite{CNN}, where $\ph$ {\black arose in the approximation} of the length of a smooth regularization of the sample paths of a Gaussian process with stationary increments.  {\blue Note that $\varphi \in \mathbb D^{1,q}$ for any $q\ge 1$, but $\ph\notin  \mathbb D^{2,2}$.}  Also, $\varphi$ has an infinite expansion \eqref{e:he} with Hermite rank $d=2$. Such a case is not covered by the findings of \cite{NZ} or  \cite{ NPP, NPR,  V} (due to the lack of sufficient regularity for the function $\varphi$), {\black and enters indeed the framework of our main result, stated in Theorem \ref{t:main}. The case of such a mapping is also covered by the recent reference \cite{nk}, where convergence in total variation is deduced for a class much smaller than $\mathbb{D}^{1,4}$, containing however $\ph(x)=|x|-\sqrt{2/\pi}$.  }}

{\black The higher regularity requirement for $\varphi$ which is necessary in \cite{NZ} stems from the method, {\black used therein}, of applying integration by parts several times. On the other hand, our approach requires that we only perform one integration by parts in the Malliavin-Stein approach, since our final estimate makes use of the intrinsic correlation bound given by Gebelein's inequality. {\black The use of Gebelein's inequality, which is the main technological breakthrough of the present paper, requires much less regularity on $\ph$}. 

\medskip

{\magenta Although the focus of our paper is on finding minimal regularity assumptions on $\varphi$ for having convergence in total variation in the Breuer-Major Theorem}, a natural question one might ask is {\black whether the rates of convergence implied by our bounds are optimal}. In view of Proposition \ref{p:pr}, applying the upper bound in Theorem \ref{t:main}-(ii) to the case $\ph=H_2$ (and $\rho\in\ell^b(\ZZ)$, for some $1\le b<2$), one obtains a rate which is not optimal. The already mentioned reference \cite{nk} shows that our results are, in general, not optimal also for the case $\ph(x)=|x|-\sqrt{2/\pi}$. Further discussions around this problem are gathered at the end of the paper --- see Section \ref{s:concl}.  

\medskip

The present paper is organised as follows.  We start by reviewing some basic elements of stochastic analysis on the Wiener space and of the Malliavin-Stein approach. Then we  introduce the new ingredient, Gebelein's inequality for correlated isonormal Gaussian processes, in Section \ref{s:2}.  We apply  a Gebelein-Malliavin-Stein bound to prove our main theorem in Section \ref{s:3}. A discussion on optimality is provided in Section \ref{s:concl}, thus concluding the paper.

\medskip

{\black Every random object considered below is defined on a common probability space $(\Omega, \mathcal F, \mathbb P)$, with $\EE$ denoting mathematical expectation with respect to $\PP$.}

\section{Preliminaries}\label{s:2}

\subsection{Stochastic analysis on the Wiener space}

The content of this subsection can be found in \cite{NPbook} or \cite{NN}. An {\it isonormal Gaussian process} $\{W(h) : h\in\frak H\}$ is a family of centered Gaussian random variables indexed by a real separable Hilbert space $\frak H$ such that the covariance satisfies 
\begin{align*}
\EE[W(g)W(h)] = \langle g,h\rangle_\frak{H}.
\end{align*}
Let $F$ be a square-integrable functional of an isonormal Gaussian process $W$. Then, $F$ has a unique Wiener-It\^o chaos expansion
\begin{align}\label{e:chaosexp}
F=\EE[F] + \sum_{k\ge 1} I_k(f_k)  \mbox{ in } L^2(\Omega),
\end{align}
where $f_k\in \frak H^{\otimes k}$ is a symmetric kernel, and $I_k(f_k)$ is the $k$-th multiple {\green Wiener-It\^o} integral, $k\ge 1$.
By convention we write $I_0(f_0)=f_0=\EE[F]$.
 By orthogonality between multiple integrals of different orders, we have $\EE[F^2]=\sum_{k\ge 0} k! \norm{f_k}_{\frak H^{\otimes k}}^2$.   Let $f:\RR^n\to\RR$ be {\black of class} $C^\infty$, and such that all its partial derivatives have at most polynomial growth. Consider a smooth functional of the form $F=f(W(h_1),...,W(h_n))$ with $h_1,..,h_n\in\frak{H}$. We define the {\it Malliavin derivative} of $F$ as 
\begin{align*}
D F= \sum_{i=1}^n  \partial_i  f(W(h_1),...,W(h_n)) h_i.
\end{align*}
The set of smooth functionals $F$ introduced above is dense in $L^q(\Omega)$, {\black $q\geq 1$}, and the operator $D$ is closable. Therefore, $D$ can be extended to 
$\mathbb{D}^{1,q}$, the set of $F$ such that there exists a sequence of smooth functionals $(F_n)_{n\ge 1}$ satisfying $\EE[|F_n-F|^q]\to 0$ and $\EE[\norm{DF_n - \eta}_\frak{H}^q]\to 0$, {\black for} some $\eta\in L^q(\Omega, \frak{H})$, {\black that we rewrite as} $\eta:=DF$. One defines similarly $D^p$ and $\mathbb D^{p,q}$. When $q=2$, these spaces are  Hilbert spaces and  we have the following characterization in terms of the chaos expansion \eqref{e:chaosexp}: $$\mathbb D^{p,2} = \{F\in L^2(\Omega): \sum_{k\ge p} k^{p} k! \norm{f_k}^2_{\frak H^{\otimes k}}<\infty \}.$$
  The adjoint of $D$, customarily called the {\it divergence operator} or the {\it Skorohod integral}, is denoted by $\de$ and satisfies the duality formula,
\begin{align}\label{duality}
\EE[\de(u)F] = \EE[\langle u, DF\rangle_\frak{H}]
\end{align} 
for all $F\in\mathbb D^{1,2}$, whenever $u: \Omega\to \frak H$ is in the domain of $\de$. The {\it Ornstein-Uhlenbeck semigroup }$(P_t)_{t\ge 0}$ is defined by Mehler's formula for all $F\in L^1(\Omega)$ by
\begin{align*}
P_t F = \EE'[F(e^{-t}W+\sqrt{1-e^{-2t}}W')],
\end{align*} 
where $W'$ is an independent copy of $W$ and $\EE'$ denotes the expectation with respect to $W'$. For $F\in L^2(\Omega)$ given by the chaos expansion \eqref{e:chaosexp}, the Ornstein-Uhlenbeck semigroup takes the form
\begin{align*}
P_t F = \sum_{k\ge 0} e^{-kt} I_k(f_k).
\end{align*}
The generator of $(P_t)_{t\ge 0}$ is denoted by $L$ and acts on the chaos expansion in a simple way,
\begin{align*}
-LF = \sum_{k\ge 1} k I_k(f_k),
\end{align*}
with $\dom L= \{F: \sum_{k\ge 1} k^2 k! \norm{f_k}_{\frak H^{\otimes k}}^2<\infty\}$. The pseudo-inverse of $L$ is defined by
\begin{align*}
-L^{-1} F=\sum_{k\ge 1} \frac{1}{k}I_k(f_k) 
\end{align*}
for all $F\in L^2(\Omega)$. We have $LL^{-1}F=F-\EE[F]$ for all $F\in L^2(\Omega)$. The key identity that links the {\black objects defined above} is  $L=-\de D$; in particular, we have $-DL^{-1}F\in \dom(\de)$
for all $F\in L^2(\Omega)$. 

We end this subsection with {\black a fundamental} product formula for multiple integrals.  
\begin{proposition}[Product formula]\label{f:product} Let $p,q$ be non-negative integers. Let $f\in \frak H^{\otimes p}$ and $g\in \frak H^{\otimes q}$ be symmetric kernels. We have
\begin{align*}
I_p(f)I_q(g) = \sum_{r=0}^{p\wedge q}r! {p\choose r} {q \choose r} I_{p+q-2r}(f\wt\otimes_r g)
\end{align*}
where $f\wt\otimes_r g$ is the symmetrized $r$-th contraction of $f$ and $g$, see \cite[p.~208]{NPbook} for a definition.  
\end{proposition}

\subsection{Malliavin-Stein approach}\label{ss:ms}

We make use of an identity ({\black labeled below as \eqref{e:key}}) first noted by Jaramillo and Nualart in \cite{JN}. 

{\black First of all, we observe that any} stationary centered Gaussian sequence $X = \{X_k : k\in \ZZ\}$ is embedded in an isonormal Gaussian process $ W =\{W(h) : h\in\frak H\}$. This means that that {\black there always exists a Hilbert space $\frak H$ and an isonormal Gaussian process $W$ (defined on the same probability space) such that, for some $\{e_k : k\ge 1\}\subset \frak H$, $W(e_k)=X_k$ for all $k$, and consequently $\E[W(e_k)W(e_l)] = \langle e_k,e_l\rangle_\frak H=\rho(k-l)$, for all $k,l$} (see, e.g., \cite[Section 1]{LN-ivan} for a justification of this fact). 

For $\ph=\sum_{\ell\ge 0}a_\ell H_\ell\in L^2(\R,\gamma)$, we define the {\black shift mapping} $\ph_1 := \sum_{\ell\ge 1}a_\ell H_{\ell-1}$ and set
\begin{align*}
u_n := \frac{1}{\sigma_n\sqrt{n}} \sum_{m=1}^n \ph_1(X_m) e_m
\end{align*}
Then,
\begin{align}\label{e:key}
\de u_n =V_n.
\end{align}
{\black To prove this, just observe that }$u_n = -DL^{-1} V_n$, and then {\black apply the relations} $L=-\de D$ and $LL^{-1} F = F$, {\black valid} for any centered random variable $F\in L^2(\Omega)$. By Stein's lemma (see \cite[Th.~3.3.1]{NPbook}) for $d_{\rm TV}$ and then by integration by parts via (\ref{duality}), we have that
\begin{align}\notag
\dtv(V_n, N(0,1)) &\le \sup_{g\in\mathcal G} |\EE[V_ng(V_n)] - \EE g'(V_n)| \\
\notag & = \sup_{g\in\mathcal G} |\EE[\de(u_n)g(V_n)] - \EE g'(V_n)| \\
\notag & = \sup_{g\in\mathcal G} |\EE g'(V_n)(1-\langle DV_n, u_n\rangle_\frak{H})| \\
\label{e:ub} &\le 2\sqrt{\Var(\langle DV_n, u_n\rangle_\frak{H})}.
\end{align}
where we used the fact that $\EE\langle DV_n, u_n\rangle_\frak{H}=\EE V_n^2=1$, and the class $\mathcal G$ is composed of those $g:\RR\to\RR$ such that $\norm{g}_\infty< \frac{\sqrt{2\pi}}{2}$ and $\norm{g'}_\infty\le 2$.

Now we estimate {\black from above} the variance in the above bound. Note that, by the chain rule and the relation $DX_k =e_k$, 
\begin{align*}
\langle DV_n, u_n\rangle_\frak{H} = \frac{1}{\sigma_n^2 n} \sum_{k,\ell=1}^n \ph'(X_k)\ph_1(X_\ell) \rho(k-\ell).
\end{align*}
Hence,
\begin{eqnarray}
\label{e:THE_variance}
&&\Var(\langle DV_n, u_n\rangle_\frak{H})\\
& =&\!\!\! \frac{1}{\sigma_n^4 n^2} \sum_{k,\ell, k',\ell'=1}^n \!\!\! \Cov(\ph'(X_k)\ph_1(X_\ell), \ph'(X_{k'})\ph_1(X_{\ell'}))\rho(k-\ell)\rho(k'-\ell').\notag\\
\notag
\end{eqnarray}

The following relation is a consequence of Meyer's inequality and of the equivalence of Sobolev norms \cite[p.72]{Nbook}, justifying our integrability assumption on $\ph$. Its proof is given  in \cite[Lem.~2.2]{NZ}.

\begin{lemma} Let $q>1$. The shift $\ph\mapsto \ph_1$ is a bounded operator from $L^{q}(\RR,\ga)$ to $L^{q}(\RR,\ga)$. 
\end{lemma} 

Note that 
\begin{align}
\sqrt{\Var(\ph'(X_k)\ph_1(X_\ell))}&\le \sqrt{\EE\ph'(X_k)^2\ph_1(X_\ell)^2}\nonumber\\
&\le \EE[\ph'(X_0)^4]^{1/4} \EE[\ph_1(X_0)^4]^{1/4}=:C(\ph)<\infty,  \label{e:C(ph)}
\end{align}
so that the covariance in \eqref{e:THE_variance} is finite. 

\subsection{Gebelein's inequality}\label{ss:gebelein}

{\black Up to some slight adaptation}, Theorem \ref{t:gengeb} can be deduced from Veraar's paper \cite{V}. For the sake of completeness, in the Appendix contained in Section \ref{s:pg} we will however present an independent proof of such a result (inspired by the approach of \cite{V}), {\black using tools and concepts that are directly connected to the framework of isonormal Gaussian processes}.

Recall that {\green an} $L^2$ functional of an isonormal Gaussian process is said to have Hermite rank $d$ if its projection to the first $d-1$ chaoses is zero, and its projection to the $d$-th chaos is non trivial. 

\begin{theorem}[Gebelein's inequality for {\green isonormal} processes]\label{t:gengeb} Let $W = \{W(h) : h\in \frak H\}$ be an isonormal Gaussian process over some real separable Hilbert space $\frak{H}$, and let $\frak{H}_1$, $\frak{H}_2$ be two Hilbert subspaces of $\frak H$. Define $W_1$ and $W_2$, respectively, to be the restriction of $W$ to $\frak H_1$ and $\frak H_2$. Now consider two measurable mappings $F_i : \R^{\frak H_i}\to \R $, $i=1,2$, and assume that each $F_i(W_i)$ is centred and square-integrable. If $F_1$ has Hermite rank equal to $p\geq 1$, one has that
\begin{eqnarray}\label{e:gengeb}
|\EE[F_1(W_1) F_2(W_2)] | \leq \theta^p \Var(F_1(W_1))^{1/2} \Var(F_2(W_2))^{1/2},
\end{eqnarray}
where $\theta := \sup_{h\in \frak H_1, g\in \frak H_2 : \|g\|, \|h\| = 1} \left| \langle h,g \rangle  \right| \in [0,1].$
\end{theorem}

\section{Proof of the main result}\label{s:3}

\subsection{$k$-sparsity}\label{s:sparse}

As we will see in the next subsection, combining Gebelein's inequality with the Malliavin-Stein approach will lead to effective upper bounds for the total variation distance in the Breuer-Major CLT. To this end, we need information on the Hermite rank of functionals of the type $F:=\ph'(W(h))\ph_1(W(g))$ for $h,g\in\frak H$ with unit norm, and $\varphi\in \mathbb D^{1,4}$.  We introduce the notion of $k$-{\it sparsity}. 

\begin{definition}
Let $\ph\in L^2(\R,\ga)$ be given by the series expansion $\ph=\sum_{q\ge d}a_q H_q$. We say the $\ph$ is $k$-{\it sparse} if $\min\{j-i: j>i\ge d, a_i\neq 0, a_j\neq 0\}\ge k$.
\end{definition}

\begin{remark}
Symmetric functions are $2$-sparse. Indeed, since $H_q(-x)=(-1)^q H(x)$ for all $q\ge 1$, the expansion of a symmetric function satisfies  $a_{2k-1}=0$ for $k\in\NN$. 
\end{remark}

\begin{lemma}
Assume that {\black $\varphi\in \mathbb D^{1,4}$} is $2$-sparse
and set $F:=\ph'(W(h))\ph_1(W(g))$, for $h,g\in\frak H$ with unit norm. Then $F-\EE[F]$ has Hermite rank at least $2$.
\end{lemma}
\begin{proof} By \cite[Th.~2.7.7]{NPbook}, we have  $H_p(W(e))=I_p(e^{\otimes p})$ {\green for $e\in\mathfrak H$} with $\norm{e}_\frak{H}=1$. Thus,
\begin{align*}
\ph'(W(h))\ph_1(W(g)) = \sum_{q\ge d}\sum_{p\ge d} qa_q a_p I_{q-1}(h^{\otimes q-1})I_{p-1}(g^{\otimes p-1}),
\end{align*}
where the series convergence in $L^2(\Omega)$.  By $2$-sparsity,  only those products of multiple integrals with indices $(p,q)$ satisfying $p=q$ or $|p-q|\ge 2$ {\green remain}. Assume  $|p-q|\ge 2$. By Proposition \ref{f:product}, the  multiple integral of lowest order in the chaos expansion for the product is $I_{|p-q|}(\cdot)$, hence the projection of $I_{q-1}(h^{\otimes q-1})I_{p-1}(g^{\otimes p-1})$ to the first chaos is zero. If $p=q$,  Proposition \ref{f:product} shows that the chaos expansion for the product contains only multiple integrals of even order, ending the proof.  
\end{proof}

\subsection{Gebelein-Malliavin-Stein upper bound}

Putting things together, we have the following Gebelein-Malliavin-Stein upper bound for the total variation distance. 

\begin{proposition}\label{p:MSG}
Let $\ph(X_1) \in \mathbb D^{1,4}$ 
 have Hermite rank $d\ge 1$, {\black and define $V_n = F_n/\sigma_n$ according to \eqref{e:F_n} and $\sigma^2_n := \Var (F_n)$}.  We have
\begin{align*}
\dtv(V_n, N(0,1)) \le \frac{4C(\ph)}{\sigma_n^2} \sqrt{\frac{1}{n^2} \sum_{i,j,k,\ell=0}^{n-1} 
\bigg| \rho(j-k) \rho(i-j)\rho(k-\ell)\bigg| .}
\end{align*}
If, in addition, $\ph$ is $2$-sparse, then
\begin{align*}
\dtv(V_n, N(0,1)) \le \frac{4C(\ph)}{\sigma_n^2} \sqrt{\frac{1}{n^2} \sum_{i,j,k,\ell=0}^{n-1}  \bigg|\rho(j-k)^2 \rho(i-j)\rho(k-\ell)\bigg| }.
\end{align*}
\end{proposition}
\begin{proof}
We evaluate the right-hand side of \eqref{e:THE_variance}, by applying Theorem \ref{t:gengeb} in the specific situation where $\frak H$ is the linear span of $\{e_i,e_j,e_k,e_\ell\}$, $\frak H_1$ the linear span of $\{e_i,e_j\}$, $\frak H_{2}$ the linear span of $\{e_k,e_\ell\}$. It is straightforward that 
\begin{align*}
|\theta| &= \max( |\rho(i-k)|, |\rho(i-\ell)|, |\rho(j-\ell)|, |\rho(j-k)|)\\
&\le  |\rho(i-k)|+ |\rho(i-\ell)|+ |\rho(j-\ell)|+|\rho(j-k)|.
\end{align*}
The conclusion follows from symmetry, and by using the estimate \eqref{e:C(ph)}. 
\end{proof}

\subsection{End of the proof}

We  are now ready to finish the proof of Theorem \ref{t:main}.
We set $\rho_n(k)=|\rho(k)|1_{|k|<n}$.
\bigskip

\noindent{\it Proof of $(i)$}.
We have
\begin{eqnarray*}
\sum_{i,j,k,\ell=0}^{n-1}  \bigg|\rho(j-k) \rho(i-j)\rho(k-\ell)\bigg|  &\le& \sum_{i,\ell=0}^{n-1} \big(\rho_n*\rho_n*\rho_n\big)(i-\ell)\\ 
&\le& n \norm{\rho_n*\rho_n*\rho_n}_{\ell^1(\ZZ)} \\
&\le& n \norm{\rho_n}^3_{\ell^1(\ZZ)}, 
\end{eqnarray*}
the last inequality being obtained by applying twice Young's inequality for convolutions.  The result follows from Proposition \ref{p:MSG}.  \qed

\bigskip

\noindent{\it Proof of $(ii)$}.
First we rewrite the sum of products as a sum of the product of convolutions by introducing the function 
$1_n(k):= 1_{|k|<n}$. We have
\begin{align*}
&\sum_{i,j,k,\ell=0}^{n-1}  |\rho(j-k)^2 \rho(i-j)\rho(k-\ell)| \\
 &= \sum_{i,j,k,\ell=0}^{n-1}  |\rho(j-k)^2 \rho(i-j)\rho(k-\ell) 1_n(\ell-i)| \\ 
&= \sum_{j,\ell=0}^{n-1}  (\rho_n* 1_n)(\ell-j) (\rho_n * \rho_n^2)(\ell-j) \le n \langle \rho_n*1_n, \rho_n*\rho^2_n\rangle_{\ell^2(\ZZ)}.
\end{align*}
Let $b\in[1,2]$. Applying successively H\"older's inequality and Young's inequality, we are led to  
\begin{align*}
&\sum_{i,j,k,\ell=0}^{n-1} \left| \rho(j-k)^2 \rho(i-j)\rho(k-\ell) \right| \\
&\le n \norm{\rho_n*1_n}_{\ell^{b\over {b-1}}(\ZZ)} \norm{\rho_n*\rho^2_n}_{\ell^b(\ZZ)} \\
&\le n \norm{\rho_n}_{\ell^{b}} \norm{1_n}_{\ell^{\frac{b}{2b-2}}(\ZZ)} \norm{\rho_n}_{\ell^{b}(\ZZ)} \norm{\rho_n^2}_{\ell^{1}(\ZZ)} = n^{\frac{3b-2}{b}} \norm{\rho_n^2}_{\ell^{1}(\ZZ)} \norm{\rho_n}^2_{\ell^b(\ZZ)}.
\end{align*}
The result follows from Proposition \ref{p:MSG}. \qed

\section{A remark on optimality}\label{s:concl}

Our Gebelein-Malliavin-Stein upper bound (Proposition \ref{p:MSG}) could not provide the rate $n^{-1/2}$ in the case where $\rho$ is square integrable but not summable. 
Indeed, restricting ourselves to the subset of indices $\{i=j=k\}$, we obtain that
\begin{eqnarray*}
&&\frac1n\sum_{i,j,k,\ell=0}^{n-1} \big| \rho(j-k)^2 \rho(i-j)\rho(k-\ell)\big| \ge \frac1n\sum_{k,\ell=1}^{n} \big|\rho(k-\ell)\big|\\
& =& 1+2\sum_{\ell=1}^{n-1}(1-\frac{\ell}{n})|\rho(\ell)| 
\ge 1+  \sum_{\ell=1}^{n/2} |\rho(\ell)| 
\end{eqnarray*}
goes to infinity as $n\to\infty$.

\section{Appendix: Proof of Theorem \ref{t:gengeb}}\label{s:pg}

We start by proving a similar result in a simpler setting, to which we can reduce the general case.

\begin{proposition}\label{p:rigid} Let $(X, Y)$ be a pair of jointly isonormal Gaussian processes over ${\frak H}$, such that $X,Y$ are rigidly correlated, in the following sense: there exists $\theta\in [-1,1]$ such that, for every $h,g\in {\frak H}$, one has $\EE[X(h)Y(g)] = \theta \langle h, g \rangle.$ Consider measurable mappings $F : \R^{{\frak H}}\to \R $ and $G : \R^{{\frak H}}\to \R $ such that $F(X)$ and $G(Y)$ are square-integrable and centred, and assume that $F$ has Hermite rank $p\geq 1$. Then,
\begin{equation}\label{e:pargeb}
|\EE[F(X)G(Y)] | \leq |\theta|^p  \Var(F(X))^{1/2}\Var(G(Y))^{1/2}
\end{equation}
\end{proposition}
\begin{proof} Let $\{e_i : i\geq 1\}$ be an orthonormal basis of ${\frak H}$. We write $\alpha, \beta,...$ to indicate multi-indices; for a multi-index $\alpha$, the symbol $H_\alpha$ indicates the corresponding multivariate polynomial. We also write 
$$
H_\alpha(X) = H_\alpha(X(e_i) : i=1,2,... )  = \prod_{i=1}^\infty H_{a_i} (X(e_i)),
$$
where $H_k$ stands for the $k$th Hermite polynomial in one variable; $H_\alpha(Y)$ is defined analogously. From the properties of Hermite polynomials and from the rigid correlation assumption, we infer that, for any choice of multi-indices $\alpha, \beta$, one has that $\E[H_\alpha(X) H_\beta(Y)] =  \theta^{|\alpha|} \alpha ! {\bf 1}_{\alpha = \beta}.$
Now, by the chaotic representation property of isonormal processes, one has that 
$$
F(X) = \sum_{\alpha : |\alpha|\geq p} b_\alpha H_\alpha(X), \quad  G(Y) = \sum_{\alpha : |\alpha|\geq 1} c_\alpha H_\alpha(Y),
$$
with convergence in $L^2(\Omega)$. By virtue of the previous discussion,
$$
|\E[F(X)G(Y)] | \leq \sum_{\alpha : |\alpha|\geq p} | b_\alpha c_\alpha| |\theta|^{|\alpha|} \alpha !\leq  |\theta| ^p\sum_{\alpha : |\alpha|\geq 1} | b_\alpha c_\alpha| \alpha !,
$$
and the conclusion follows from an application of the Cauchy-Schwarz inequality.
\end{proof}

We now turn to the proof of  Theorem \ref{t:gengeb}.

\begin{proof}[Proof of Theorem \ref{t:gengeb}] Without loss of generality, we assume that $\theta \in (0,1)$. For $i=1,2$, we denote by $\pi_{{\frak H}_i}$ the orthogonal projection operator onto ${\frak H}_i$. We will make use of the following estimate: for every $g\in {\frak H}_2$ with unit norm,
\begin{equation}\label{e:esti}
\| \pi_{{\frak H}_1}(g) \| \leq \theta,
\end{equation}
which follows from the relation $\| \pi_{{\frak H}_1}(g) \|^2= | \langle g, \pi_{{\frak H}_1}(g) \rangle |\leq \theta \| \pi_{{\frak H}_1}(g) \|.$ {\black Now write ${\frak H}_1\oplus {\frak H}_2$ to indicate the direct sum of ${\frak H}_1$ and ${\frak H}_2$}. The key of the proof is the definition of mappings $\tau_1:\frak H_2\to\frak H_1$,  $\tau_2:\frak H_2\to\frak H_2$, and $\tau : {\frak H}_2\to  {\frak H}_1\oplus {\frak H}_2$ given by $g \mapsto \tau_1(g)\oplus \tau_2(g) $, with the following two properties:
\begin{enumerate}
\item[(i)] for $h\in {\frak H}_1$ and $g\in {\frak H}_2$, $\langle h , \tau_1(g) \rangle = \theta^{-1} \langle h,g\rangle$;
\item[(ii)]  $\tau$  verifies the isometric property: $\langle \tau(g) , \tau(k) \rangle_{{\frak H}_1\oplus{\frak H}_2} = \langle g,k\rangle_{{\frak H}}$, for every $g,k \in {\frak H}_2$.
\end{enumerate}
In order to define such a mapping $\tau$, we first observe that, by virtue of \eqref{e:esti}, the positive self-adjoint and bounded operator $U$, from ${\frak H}_2$ into itself, given by $g\mapsto U(g) = \pi_{{\frak H}_2} (\pi_{{\frak H}_1} (g) ),$
is such that
$$
\|U\|_{op} = \sup_{g,k \in {\frak H}_2 : \|g\|,\|k\|=1} | \langle U(g) , k\rangle| \leq \theta \sup_{g \in {\frak H}_2 : \|g\|=1} \| \pi_{{\frak H}_1}(g) \|,
$$
and therefore $\|U\|_{op}\leq \theta^2$, by virtue of \eqref{e:esti}. This implies that the operator $V(g) := \sqrt{ {\rm Id} - U/\theta^2} $ is well-defined. In particular one checks that a mapping $\tau$ satisfying the two properties (i) and (ii) listed above is given by $\tau_1(g ) = \theta^{-1} \pi_{{\frak H}_1}(g) $ and $\tau_2(g) = V(g)$, for every $g\in {\frak H}_2$. We now consider two auxiliary independent isonormal Gaussian processes $Y,Z$ over ${\frak H}_1\oplus{\frak H}_2$, and we set $R := \theta Y +\sqrt{1-\theta^2} Z,$
in such a way that $Y,R$ are rigidly correlated with parameter $\theta$, in the sense made clear in the statement of Proposition \ref{p:rigid}. It is also easily verified that, by a direct covariance computation and with obvious notation,
$$
\big( Y({\frak H}_1\oplus\{0\}) , R(\tau({\frak H}_2))  \big) \stackrel{\rm law}{=} (X_1, X_2).
$$
To conclude the proof, we apply Proposition \ref{p:rigid} as follows:
\begin{align*}
|\E[F_1(X_1)F_2(X_2)] | &= |\E[F_1(Y({\frak H}_1\oplus\{0\}))F_2(R(\tau({\frak H}_2)))] |\\ 
& \leq \theta^p \Var(F_1(X_1))^{1/2} \Var(F_2(X_2))^{1/2},
\end{align*}
where we have used the fact that $F_1(Y({\frak H}_1\oplus\{0\}))$ has also Hermite rank $p$, as well as the relations $\Var(F_1(X_1)) = \Var(F_1(Y({\frak H}_1\oplus\{0\}))),$ and $\Var(F_2(X_2)) = \Var(F_2(R(\tau({\frak H}_2)))).$
\end{proof}

\smallskip

\noindent{\bf Acknowledgments}. We are grateful to David Nualart 
for stimulating discussions on this topic. Ivan Nourdin is supported by the FNR Grant R-AGR-3585-10 (APOGee) at the University of Luxembourg. Giovanni Peccati is supported by the FNR Grant R-AGR-3376-10 (FoRGES) at the University of Luxembourg. 
Xiaochuan Yang is supported by the FNR Grant R-AGR-3410-12-Z (MiSSILe) at Luxembourg and Singapore Universities.

\bibliographystyle{plain}




\end{document}